\documentclass{amsart}
\usepackage{amssymb}
\usepackage{amssymb,amsmath}
\usepackage{mathrsfs}

\input xy
\xyoption{all}
\newtheorem{theorem}{Theorem}[section]
\newtheorem{lemma}[theorem]{Lemma}
\newtheorem{problem}[theorem]{Problem}
\newtheorem{example}[theorem]{Example}

\newcommand{\I}{\mathcal{I}}
\newcommand{\E}{\mathcal{E}}
\newcommand{\U}{\mathcal{U}}
\newcommand{\V}{\mathcal{V}}
\newcommand{\W}{\mathcal{W}}
\newcommand{\A}{\mathcal{A}}
\newcommand{\M}{\mathcal{M}}
\newcommand{\C}{\mathcal{C}}
\newcommand{\F}{\mathcal{F}}
\newcommand{\HH}{\mathcal{H}}
\newcommand{\id}{\mathrm{id}}
\newcommand{\cl}{\mathrm{cl}}

\newcommand{\IQ}{\mathbb{Q}}
\newcommand{\IN}{\mathbb{N}}
\newcommand{\w}{\omega}
\newcommand{\e}{\varepsilon}
\newcommand{\diam}{\mathrm{diam}}
\newcommand{\mesh}{\mathrm{mesh}}
\newcommand{\Clop}{\mathrm{Clop}}

\def\w{\omega}

\title[Classifying invariant $\sigma$-ideals on Cantor measure space]{Classifying invariant $\sigma$-ideals with analytic base\\ on good Cantor measure spaces}

\author{Taras Banakh,  Robert Ra\l owski, Szymon \.Zeberski}
\address[Taras Banakh]{Department of Mathematics, Ivan Franko National University of Lviv, Ukraine, and
Institute of Mathematics, Jan Kochanowski University, Kielce, Poland}
\email{t.o.banakh@gmail.com}
\address[R.Ra\l owski, S.\.Zeberski]{Institute of Mathematics and Computer Science, Wroc\l aw University of Technology, Wroc\l aw, Poland}
\email{robert.ralowski@pwr.wroc.pl, szymon.zeberski@pwr.wroc.pl}
\subjclass{03E15; 28A05; 28D05; 54H05}
\keywords{good Cantor measure space, measure-preserving homeomorphism, invariant $\sigma$-ideal}

\begin{document}

\maketitle
\footnotetext{\noindent
The work has been partially financed by NCN means granted by decision DEC-2011/01/B/ST1/01439. }

\begin{abstract} Let $X$ be a zero-dimensional compact metrizable space endowed with a strictly positive continuous Borel $\sigma$-additive measure $\mu$ which is good in the sense that for any clopen subsets $U,V\subset X$ with $\mu(U)<\mu(V)$ there is a clopen set $W\subset V$ with $\mu(W)=\mu(U)$. We study $\sigma$-ideals with Borel base on $X$ which are invariant under the action of the group $\HH_\mu(X)$ of measure-preserving homeomorphisms of $(X,\mu)$, and show that any such $\sigma$-ideal $\I$ is equal to one of seven $\sigma$-ideals: $\{\emptyset\}$, $[X]^{\le\w}$, $\mathcal E$, $\M\cap\mathcal N$, $\M$, $\mathcal N$, or $[X]^{\le \mathfrak c}$. Here $[X]^{\le\kappa}$ is the ideal consisting of subsets of cardiality $\le\kappa$ in $X$, $\M$ is the ideal of meager subsets of $X$, $\mathcal N=\{A\subset X:\mu(A)=0\}$ is the ideal of null subsets of $(X,\mu)$, and $\mathcal E$ is the $\sigma$-ideal generated by closed null subsets of $(X,\mu)$.
\end{abstract}

\section{Introduction and Main Result}

The investigation of (topologically) invariant $\sigma$-ideals with Borel or analytic base on some model topological spaces was initiated by the authors in \cite{BMRZ1} and \cite{BMRZ2}. In this paper we shall study and classify invariant $\sigma$-ideals with analytic base on (good) Cantor measure spaces.

A {\em measure space} is a pair $(X,\mu)$ consisting of a topological space $X$ and a $\sigma$-additive measure $\mu:\mathcal B(X)\to[0,\infty)$ defined on the $\sigma$-algebra of Borel subsets of $X$. A measure space $(X,\mu)$ is called a {\em Cantor measure space} if the topological space $X$ is homeomorphic to the Cantor cube $\{0,1\}^\w$ and the measure $\mu$ is {\em continuous} in the sense that $\mu(\{x\})=0$ for any point $x\in X$. By Brouwer's Theorem \cite[7.4]{Ke}, a topological space $X$ is homeomorphic to the Cantor cube $\{0,1\}^\w$ if and only if $X$ is a zero-dimensional compact metrizable space without isolated points. In this case $X$ is called  {\em a Cantor set}.

For a measure space $(X,\mu)$ the set $$\mu(\Clop(X))=\{\mu(U):\mbox{$U$ is clopen in $X$}\}$$is called the {\em clopen values set} of $(X,\mu)$. By a {\em clopen} set in a topological space $X$ we understand a subset which is simultaneously closed and open in $X$. A topological space $X$ is {\em zero-dimensional} if clopen sets form a base of the topology of $X$.

Two measure spaces $(X,\mu)$ and $(Y,\lambda)$ are defined to be {\em isomorphic} if there exists a measure-preserving homeomorphism $h:X\to Y$.
The measure-preserving property of $h$ means that $\lambda(h(B))=\mu(B)$ for any Borel subset $B\subset X$. It is clear that for any two isomorphic measure spaces $(X,\mu)$ and $(Y,\lambda)$ we get $\mu(\Clop(X))=\lambda(\Clop(Y))$. In \cite{Akin} Akin proved that for good Cantor measure spaces the converse implication holds. A Cantor measure space $(X,\mu)$ is called {\em good} if its measure $\mu$ is {\em good} in the sense of Akin \cite{Akin}, i.e., $\mu$ is continuous, {\em strictly positive} (which means that $\mu(U)>0$ for any non-empty open set $U\subset X$), and satisfies the {\em Subset Condition} which means that for any clopen sets $U,V\subset X$ with $\mu(U)<\mu(V)$ there is a clopen set $U'\subset V$ such that $\mu(U')=\mu(U)$. According to Akin's Theorem~2.9 \cite{Akin}, two good Cantor measure spaces $(X,\mu)$ and $(Y,\lambda)$ are isomorphic if and only if $\mu(\Clop(X))=\lambda(\Clop(Y))$.

The class of good Cantor measure spaces includes all infinite compact metrizable zero-dimensional topological groups $G$ endowed with the Haar measure. Moreover, by Theorem~2.16 \cite{Akin}, a Cantor measure space $(X,\mu)$ is isomorphic to a (monothetic) compact topological group $G$ endowed with the Haar measure if and only if $(X,\mu)$ is good and $1\in \mu(\Clop(X))\subset\IQ\cap[0,1]$.

The aim of this paper is to study and classify invariant $\sigma$-ideals with Borel (or analytic) base on (good) Cantor measure spaces.

A family $\I$ of subsets of a set $X$ is called an {\em ideal} on $X$ if for any sets $A,B\in\I$ and $C\subset X$ we get $A\cup B\in\I$ and $A\cap C\in\I$. An ideal $\I$ on $X$ is a {\em $\sigma$-ideal} if $\bigcup\A\in\I$ for any countable subfamily $\A\subset\I$.
A subfamily $\mathcal B\subset\I$ is called a {\em base} for the ideal $\I$ if each set $A\in\I$ is contained in some set $B\in\mathcal B$.
Each family $\mathcal B$ of subsets of $X$ generates a $\sigma$-ideal $$\langle\mathcal B\rangle=\{A\subset X:\mbox{$A\subset \bigcup\C$ for some countable subfamily $\C\subset\mathcal B$}\}.$$ This is the smallest $\sigma$-ideal containing the family $\mathcal B$.

Observe that for every infinite cardinal $\kappa$ the family $[X]^{\le\kappa}=\{A\subset X:|A|\le\kappa\}$ of subsets of cardinality $\le\kappa$ is a $\sigma$-ideal on $X$. For $\kappa=0$ the ideal $[X]^{\le 0}=\{\emptyset\}$ is a $\sigma$-ideal, too.

We shall say that an ideal $\I$ on a topological space $X$ has {\em Borel} ({\em analytic}) {\em base} if $\I$ has a base $\mathcal B\subset\I$ consisting of Borel (analytic) subsets of $X$. Let us recall that a subset $A\subset X$ is called {\em analytic} if $A$ is a continuous image of a Polish space. Since each Borel subset of a Polish space is analytic \cite[13.7]{Ke}, each ideal with Borel base on a Polish space has an analytic base.

Observe that for any topological space $X$ the ideals $[X]^{\le0}$, $[X]^{\le\w}$ and $[X]^{\le|X|}$ have Borel base.

An ideal $\I$ on a measure space $(X,\mu)$ will be called {\em invariant} if for any measure-preserving homeomorphism $h:X\to X$ we get $\{h(A):A\in\I\}=\I$. Measure-preserving homeomorphisms of $(X,\mu)$ form a group $\HH_\mu(X)$ called the {\em automorphism group} of the measure space $(X,\mu)$. It is a subgroup of the homeomorphism group $\HH(X)$ of the topological space $X$.

On each measure space $(X,\mu)$ consider the following four invariant $\sigma$-ideals with Borel base:
\begin{itemize}
\item the $\sigma$-ideal $\M$ of meager subsets of $X$ (it is  generated by closed nowhere dense subsets of $X$);
\item the $\sigma$-ideal $\mathcal N=\{A\subset X:\mu(A)=0\}$ of null subsets of $(X,\mu)$ (it is generated by Borel subsets of zero $\mu$-measure);
\item the $\sigma$-ideal $\M\cap\mathcal N$ of meager null subsets of $(X,\mu)$;
\item the $\sigma$-ideal $\mathcal E$ generated by closed null subsets of $(X,\mu)$.
\end{itemize}

The $\sigma$-ideals $\M$, $\mathcal N$, $\M\cap\mathcal N$ and $\mathcal E$ play an important role in Set Theory and its Applications, see \cite{BJ}. For a measure space $(X,\mu)$ endowed with a continuous strictly positive Borel measure $\mu$ these ideals relate as follows:
$$
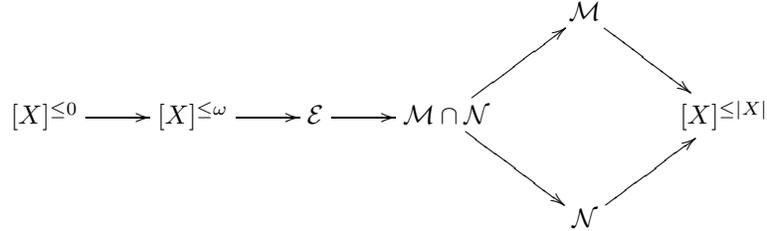
\xymatrix{
&&&&\mathcal M\ar[rd]\\
[X]^{\le0}\ar[r]&[X]^{\le\w}\ar[r]&\mathcal E\ar[r]&\M\cap\mathcal N\ar[ru]\ar[rd]&&[X]^{\le|X|}\\
&&&&\mathcal N\ar[ru]
}
$$
(in the diagram an arrow $\I\to\mathcal J$ indicates that $\I\subset\mathcal J$).

It turns out that on good Cantor measure spaces these seven $\sigma$-ideals exhaust all possible invariant $\sigma$-ideals with analytic base.

\begin{theorem}\label{main} Each invariant $\sigma$-ideal with analytic base on a good Cantor measure space $(X,\mu)$ is equal to one of the $\sigma$-ideals: $[X]^{\le 0}$, $[X]^{\le\w}$, $\mathcal E$, $\M\cap\mathcal N$, $\M$, $\mathcal N$, or $[X]^{\le\mathfrak c}$.
\end{theorem}

This theorem will be proved in Section~\ref{s4} after some preparatory work made in Sections~\ref{s2} and \ref{s3}. It should be mentioned that Theorem~\ref{main} is specific for good Cantor measure spaces and cannot be generalized to any  Cantor measure space.

\begin{example} There is a Cantor measure space $(X,\mu)$ having trivial automorphism group $\HH_\mu(X)$. On this measure space each ideal is invariant. Consequently, the Cantor measure space $(X,\mu)$ supports $2^{\mathfrak c}$ pairwise distinct invariant $\sigma$-ideals with Borel base.
\end{example}

\begin{proof} By \cite{Akin} there exists a Cantor measure space $(X,\mu)$ with trivial automorphism group $\HH_\mu(X)$. Choose a family $\mathcal C$ of cardinality continuum consisting of pairwise disjoint Cantor sets in $X$. Each subfamily $\A\subset\C$ generates an invariant $\sigma$-ideal $\langle\A\rangle$ with Borel base (more precisely, a base consisting of $F_\sigma$-subsets of $X$). It is clear that $\langle\A\rangle\ne\langle\mathcal B\rangle$ for any distinct subfamilies $\A\ne\mathcal B$ in $\C$. Since the family $\mathcal C$ contains $2^{\mathfrak c}$ subfamilies, the measure space $(X,\mu)$ supports $2^{\mathfrak c}$  invariant $\sigma$-ideals $\langle\A\rangle$, $\A\subset\C$, with Borel base.
\end{proof}

A measure space $(X,\mu)$ is called {\em minimal} if for each point $x\in X$ its orbit $\{h(x):h\in\HH_\mu(X)\}$ is dense in $X$. If the orbit $\{h(x):h\in\HH_\mu(X)\}$ of any point $x\in X$ coincides with $X$, then the measure space $(X,\mu)$ is called {\em homogeneous}. It is clear that each homogeneous measure space is minimal. In \cite{Akin} Akin proved that each good Cantor measure space is homogeneous. On the other hand, there are examples \cite{BH} of minimal Cantor measure spaces, which are not good. We do not know if each homogeneous Cantor measure space is good. Because of that the following problem does not reduce to Theorem~\ref{main}.

\begin{problem} Classify invariant $\sigma$-ideals with Borel base on a homogeneous Cantor measure space $(X,\mu)$. Is any non-trivial invariant $\sigma$-ideals with Borel base equal to one of seven standard ideals: $[X]^{\le0}$, $[X]^{\le\w}$, $\mathcal E$, $\M\cap\mathcal N$, $\M$, $\mathcal N$ or $[X]^{\le\mathfrak c}$?
\end{problem}

\section{Homogeneity properties of good Cantor measure spaces}\label{s2}

In this section we shall establish some homogeneity properties of good Cantor measure spaces. First we fix some notation.

Let $(X,d)$ be a metric space. For a point $x\in X$ and real number $\e>0$ by $O_\e(x)=\{y\in X:d(y,x)<\e\}$ we denote the open $\e$-ball centered at $x$. For a subset $A\subset X$ let $O_\e(A)=\bigcup_{a\in A}O_\e(a)$ be its $\e$-neighborhood and $\diam(A)=\sup(\{0\}\cup\{d(x,y):x,y\in A\})$ be its diameter.
By $\cl_X(A)$ or just $\bar A$ we shall denote the closure of a subset $A\subset X$ in $X$.

For a family $\A$ of subsets of a metric space we put $\mesh\A=\sup_{A\in\A}\diam(A)$. A family $\A$ of sets is {\em disjoint} if $A\cap B=\emptyset$ for any distinct sets $A,B\in\A$.
By a {\em partition} of a set $X$ we understand any disjoint family $\A$ of sets such that $\bigcup\A=X$.
 A family $\A$ of subsets of a topological space $X$ is called {\em clopen} if each set $A\in\A$ is non-empty and clopen in $X$.

A topological space $X$ is called {\em zero-dimensional} if clopen sets form a base of the topology of $X$. It is well-known that for any cover $\U$ of a zero-dimensional separable metrizable space $X$ there is a clopen partition $\V$ of $X$ {\em inscribed} into $\U$ in the sense that each set $V\in\V$ is contained in some set $U\in\U$.

For a Borel measure $\mu:\mathcal B(X)\to[0,\infty)$ defined on the $\sigma$-algebra $\mathcal B(X)$ of Borel subsets of $X$ and any Borel subset $Y\subset X$ by $\mu|Y$ we shall denote the restriction of $\mu$ to the $\sigma$-algebra $\mathcal B(Y)\subset\mathcal B(X)$ of Borel subsets of $Y$. For a Borel measure $\mu$ on a topological space $X$ the {\em support} $\mathrm{supp}(\mu)$ of $\mu$ is the (closed) set of all points $x\in X$ such that each neighborhood $O_x\subset X$ of $x$ has measure $\mu(O_x)>0$. It is well-known that each $\sigma$-additive Borel measure $\mu$ on a compact metrizable space $X$ is {\em regular} in the sense that for any Borel subset $B\subset X$ and any $\e>0$ there is a closed subset $K\subset B$ of $X$ such that $\mu(B\setminus K)<\e$.

For future references let us write down a fundamental result of Akin \cite{Akin}.

\begin{lemma}[Akin]\label{akin} Two good Cantor measure spaces $(X,\mu)$ and $(Y,\lambda)$ are isomorphic if and only if $\mu(\Clop(X))=\lambda(\Clop(Y))$.
\end{lemma}

We shall need the following improvement of the Subset Condition in good Cantor measure spaces.

\begin{lemma}\label{scond} Let $(X,\mu)$ be a good Cantor measure space, $U\subset X$ be a clopen set and $K\subset U$ be a compact subset. For every $\alpha\in \mu(\Clop(X))$ with $\mu(K)<\alpha\le \mu(U)$ there is a clopen subset $V\subset U$ such that $K\subset V$ and $\mu(V)=\alpha$.
\end{lemma}

\begin{proof} If $\alpha=\mu(U)$, then we can put $V=U$. So, assume that $\alpha<\mu(U)$. By the regularity of the measure $\mu$ and the zero-dimensionality of the space $X$, there is a clopen neighborhood $O_K\subset U$ of $K$ such that $\mu(O_K)<\alpha$. Since $\alpha\in\mu(\Clop(X))$, there is a clopen set $A\subset X$ with $\mu(A)=\alpha$. By the Subset Condition, the set $A$ contains a clopen subset $B\subset A$ of measure $\mu(B)=\mu(O_K)$. Then the clopen set $A\setminus B$ has measure $\beta=\mu(A\setminus B)=\alpha-\mu(O_K)$. Since $\mu(U\setminus O_K)>\alpha-\mu(O_K)=\beta$, by the Subset Condition, the clopen set $U\setminus O_K$ contains a clopen set $W$ of measure $\mu(W)=\beta$. Then the clopen set $V=O_K\cup W\subset U$ contains $K$ and has measure $\mu(V)=\mu(O_K)+\mu(W)=\alpha$.
\end{proof}

We shall need the following well-known homogeneity property of the Cantor cube (which is attributed to Ryll-Nardzewski in \cite{KR}).

\begin{lemma}\label{extop} Any homeomorphism $f:A\to B$ between closed nowhere dense subsets $A,B\subset X$ of the Cantor cube $X=\{0,1\}^\w$ extends to a homeomorphism $\bar f:X\to X$ of $X$.
\end{lemma}

Our next lemma can be considered as a measure-topological counterpart of Lemma~\ref{extop}. It implies the transitivity of good Cantor measure spaces, proved by Akin in \cite[2.9]{Akin}.

\begin{lemma}\label{exmes} Any measure-preserving homeomorphism
$f:A\to B$ between closed nowhere dense subsets $A,B\subset X$ of a good Cantor measure space $(X,\mu)$ extends to a measure-preserving homeomorphism $f:X\to X$ of $X$.
\end{lemma}

\begin{proof} Fix any metric $d$ on the space $X$ generating the topology of $X$.

Let $\U_0=\V_0=\{X\}$ and $\xi_0:\U_0\to\V_0$ be the unique bijective map. By induction, for every $n\in\IN$ we shall construct two finite clopen partitions $\U_n$ and $\V_n$ of $X$, a bijective map $\xi_n:\U_n\to\V_n$, and two maps $p_n:\U_n\to\U_{n-1}$ and $\pi_n:\V_n\to\V_{n-1}$ such that the following conditions are satisfied:
\begin{itemize}
\item[$(1_n)$] $\mesh\U_n<1/2^n$ and $\mesh\V_n<1/2^n$;
\item[$(2_n)$] for any sets $U\in\U_n$ and $V\in\V_n$ we get  $U\subset p_n(U)\in\U_{n-1}$ and $V\subset\pi_n(V)\in\V_{n-1}$;
\item[$(3_n)$] $\pi_{n}\circ\xi_{n}=\xi_{n-1}\circ p_{n}$;
\item[$(4_n)$] $\mu(\xi_n(U))=\mu(U)$ for any set $U\in\U_n$;
\item[$(5_n)$] $f(U\cap A)=\xi_n(U)\cap B$ for every set $U\in\U_n$.
\end{itemize}

Assume that for some $n\in\IN$ and all $k<n$ the partitions $\U_k$, $\V_k$ and maps $\xi_k:\U_k\to\V_k$, $p_k:\U_k\to\U_{k-1}$, $\pi_k:\V_k\to\V_{k-1}$ satisfying the conditions $(1_k)$--$(5_k)$ have been constructed.

Fix any subset $U\in\U_{n-1}$ and consider the clopen set $V=\xi_{n-1}(U)\in\V_{n-1}$. By the inductive assumption $(5_{n-1})$, $f(A\cap U)=B\cap V$. By Lemma~\ref{extop}, the homeomorphism $f|A\cap U:A\cap U\to B\cap V$ extends to a homeomorphism $\bar f_U:U\to V$. Choose a disjoint clopen partition $\W_U$ of the space $U$ such that $\mesh\W_U<2^{-n}$ and $\mesh \bar f_U(\W_U)<2^{-n}$ where $\bar f_U(\W_U)=\{\bar f_U(W):W\in\W_U\}$. Let $\W'_U=\{W\in\W_U:W\cap A\ne\emptyset\}$. Applying Lemma~\ref{scond}, for every set $W\in \W'_U$, choose two clopen sets $U_W,V_W\subset X$ such that $W\cap A\subset U_W\subset W$, $f(W\cap A)\subset V_W\subset f(W)$ and $\mu(U_W)=\mu(V_W)$. It follows that the sets
$U^\circ=U\setminus\bigcup_{W\in\W_U'}U_W$ and $V^\circ=V\setminus\bigcup_{W\in\W_U'}V_W$ have the same measure $\mu(U^\circ)=\mu(V^\circ)>0$. By Lemma~\ref{akin}, the good Cantor measure spaces $(U^\circ,\mu|U^\circ)$ and $(V^\circ,\mu|V^\circ)$ are isomorphic. Consequently, there exists a measure-preserving homeomorphism $h_U:U^\circ\to V^\circ$.
Choose any clopen partition $\C_U$ of the Cantor set $U^\circ$ such that $\mesh(\C_U)<2^{-n}$ and $\mesh(h_U(\C_U))<2^{-n}$. Let $\U_n(U)=\C_U\cup\{U_W:W\in \W'(U)\}$, $\V_n(U)=h_U(\C_U)\cup\{V_W:W\in\W'(U)\}$ and define the bijective map $\xi_{U,n}:\U_n(U)\to\V_n(U)$ by the formula
$$\xi_{U,n}(U')=\begin{cases}V_W&\mbox{if $U'=U_W$ for some $W\in\W'(U)$},\\
h_U(U')&\mbox{if $U'\in\C_U$}.
\end{cases}
$$

Now consider the clopen partitions $\U_n=\bigcup_{U\in\U_{n-1}}\U_n(U)$ and $\V_n=\bigcup_{U\in\U_{n-1}}\V_n(U)$ of $X$ and let $\xi_n:\U_n\to\V_n$ be the bijective function such that $\xi_n|\U_n(U)=\xi_{U,n}$ for any $U\in\U_{n-1}$.

Let $p_n:\U_n\to\U_{n-1}$ be the map assigning to each set $U'\in\U_n$ the unique set $U\in\U_{n-1}$ such that $U'\in\U_n(U)$ and hence $U'\subset U$. Also consider the map $\pi_n:\V_n\to\V_{n-1}$ assigning to each set $V'\in\V_n$ the unique set $V\in\V_{n-1}$ that contains $V'$. It is easy to see that the partitions $\U_n$, $\V_n$, and functions $\xi_n$, $p_n$, and $\pi_n$ satisfy the conditions $(1_n)$--$(5_n)$.

After completing the inductive construction, define a measure-preserving homeomorphism $\bar f:X\to X$ assigning to each point $x\in X$ the unique point of the intersection $\bigcap_{n\in\w}\{\xi_n(U):x\in U\in\U_n\}$. The conditions $(5_n)$, $n\in\IN$, guarantee that $\bar f|A=f$.
\end{proof}

\begin{lemma}\label{haplyk} Let $(X,\mu)$, $(Y,\lambda)$ be Cantor measure spaces such that $\mu(X)<\lambda(Y)$ and the measure $\lambda$ is strictly positive. Let $G_X\subset X$ and $G_Y\subset Y$ be two $G_\delta$-sets of measure $\mu(G_X)=\lambda(G_Y)=0$ such that $G_Y$ is dense in $Y$. Then there is a measure-preserving embedding $f:X\to Y$ such that $f(G_X)\subset G_Y$.
\end{lemma}

\begin{proof}
Let $d_X,d_Y$ be metrics of diameter $<1$ generating the topologies of the spaces $X,Y$, respectively. Write the complements $X\setminus G_X$ and $Y\setminus G_Y$ as countable unions $X\setminus G_X=\bigcup_{n\in\w}A_n$ and $Y\setminus G_Y=\bigcup_{n\in\w}B_n$ of increasing sequences $(A_n)_{n\in\w}$ and $(B_n)_{n\in\w}$ of compact sets.

Let $U_0=X$ and choose any clopen subset $V_0\subset Y$ such that $0<\lambda(V_0)<1$ if $\mu(U_0)=0$ and  $\mu(U_0)<\lambda(V_0)<2\cdot \mu(U_0)$ if $\mu(U_0)>0$. Let $\U_0=\{U_0\}$, $\V_0=\{V_0\}$ and $\xi_0:\U_0\to\V_0$ be the unique bijective map. Since $V_0\setminus G_A=\bigcup_{m\in\w}V_0\cap B_m$ and $\lambda(V_0\setminus G_A)=\lambda(V_0)>\mu(U_0)$, there is a number $m_0\in\w$ such that $\mu(V_0\cap B_{m_0})>\mu(U_0)$. Since the set $B_{m_0}$ is nowhere dense in $Y$ and the measure $\lambda$ on $Y$ is strictly positive, $\lambda(V_0\setminus B_{m_0})>0=\mu(U_0\cap G_X)=\mu\big(\bigcap_{k\in\w}U_0\setminus A_k\big)$ and hence we can find a number $k_0\in\w$ such that $\mu(U_0\setminus A_{n_0})<\lambda(V_0\setminus B_{m_0})$.

By induction for every $n\in\IN$ we shall construct two numbers $k_n,m_n\in\IN$, a clopen partition $\U_n$ of $X$, a finite disjoint family $\V_n$ of clopen subsets of the space $Y$, a
bijective function $\xi_n:\U_n\to\V_n$, and two functions $p_{n}:\U_{n}\to\U_{n-1}$ and $\pi_{n}:\V_{n}\to\V_{n-1}$ such that the following conditions are satisfied:
\begin{itemize}
\item[$(1_n)$] $\mesh\U_n<1/2^n$ and $\mesh\V_n<1/2^n$;
\item[$(2_n)$] for any sets $U\in\U_n$ and $V\in\V_n$ we get  $U\subset p_n(U)\in\U_{n-1}$ and $V\subset\pi_n(V)\in\V_{n-1}$;
\item[$(3_n)$] $\pi_{n}\circ\xi_{n}=\xi_{n-1}\circ p_{n}$;
\item[$(4_n)$] $\mu(U)<\lambda(\xi_n(U))<(1+2^{-n})\cdot\mu(U)$ for any set $U\in\U_n$ with $\mu(U)>0$;
\item[$(5_n)$] $0<\lambda(\xi_n(U))<1/(2^n\cdot|\U_n|)$ for any set $U\in\U_n$ with $\mu(U)=0$;
\item[$(6_n)$] for every $U\in\U_{n}$ with $U\cap A_{k_{n-1}}=\emptyset$, we get $\xi_{n}(U)\cap B_{m_{n-1}}=\emptyset$;
\item[$(7_n)$] $\mu(U)<\lambda(\xi_n(U)\cap B_{m_n})$ for every $U\in\U_n$;
\item[$(8_n)$] $\mu(U\setminus A_{k_n})<\lambda(\xi_n(U)\setminus B_{m_n})$ for every $U\in\U_n$;
\item[$(9_n)$] $k_{n}>k_{n-1}$ and $m_{n}>m_{n-1}$.
\end{itemize}

Assume that for some $n\in\IN$ and every $k<n$ we have constructed  families $\U_{k}$, $\V_{k}$, and maps $\xi_k:\U_k\to\V_k$, $p_k:\U_k\to\U_{k-1}$, $\pi_k:\V_k\to\V_{k-1}$ satisfying the conditions $(1_k)$--$(9_k)$.

For every set $U\in\U_{n-1}$ and its image $V=\xi_{n-1}(U)\in\V_{n-1}$ we shall construct a clopen partition $\U_n(U)$ of $U$, a disjoint family $\V'_n(V)$ of clopen subsets of $V$ and a bijective map $\xi'_{U,n}:\U_n(U)\to\V_n'(V)$ such that the following conditions are satisfied:
\begin{itemize}
\item[$(1_{U,n})$] $\mesh\U_n(U)<1/2^n$ and $\mesh\V'_n(V)<1/2^n$;
\item[$(4_{U,n})$] $\mu(U')<\lambda(\xi'_{U,n}(U'))<(1+2^{-n})\cdot\mu(U')$ for every set $U'\in\U_n(U)$ with $\mu(U')>0$;
\item[$(6_{U,n})$] for every $U'\in\U_{n}(U)$ with $U'\cap A_{k_{n-1}}=\emptyset$ we get $\xi'_{U,n}(U')\cap B_{m_{n-1}}=\emptyset$.
\end{itemize}

Fix any set $U\in\U_{n-1}$ and consider its image $V=\xi_{n-1}(U)\in\V_{n-1}$. 

Choose any clopen partition $\U'_n(U)$ of $U$ such that $\mesh\,\U'_n(U)<2^{-n}$. Consider the subfamily  $\U''_n(U)=\{U'\in\U'_n(U):U'\cap A_{k_{n-1}}=\emptyset\}$ and observe that $\mu\big(\bigcup\U''_n(U)\big)\le\mu(U\setminus A_{k_{n-1}})<\lambda(V\setminus B_{m_{n-1}})$ according to the condition $(8_{n-1})$.

Choose a positive real number $\e_U$ such that
\begin{itemize}
\item $\e_U<\min\big(\{2^{-n}\}\cup\{\mu(U')/2^n:U'\in\U_n'(U),\;\mu(U')>0\}\big)$;
\item $\mu(U\setminus A_{k_{n-1}})<\lambda(V\setminus O_{\e_U}(B_{m_{n-1}}))$,
\item $\e_U<\big(\lambda(V\setminus O_{\e_U}(B_{m_{n-1}}))-\mu(U\setminus A_{k_{n-1}})\big)/|\U'_n(U)|$, and
\item $\e_U<(\lambda(V)-\mu(U))/|\U'(U)|$.
\end{itemize}

Next, choose any clopen partition $\V'_{n}(V)$ of $V$ such that each set $V'\in\V'_{n}(V)$ has diameter $\diam(V')<\e_U$ and measure $\lambda(V')<\e_U$. Let $\V''_n(V)=\{V'\in\V'_n(V):V'\cap B_{m_{n-1}}=\emptyset\}$ and observe that
$$V\setminus O_{\e_U}(B_{m_{n-1}})\subset \textstyle{\bigcup\V''_n(V)}\subset V\setminus B_{m_{n-1}}.$$

To every set $U'\in\U'_{n}(U)$ assign a subfamily $\V'_{n}(U')\subset \V'_{n}(V)$ such that
$$\mu(U')<\lambda\big(\textstyle{\bigcup}\V'_{n}(U')\big)<\mu(U')+\e_U,$$ the families $\V'_{n}(U')$, $U'\in\U'_{n}(U)$, are pairwise disjoint, and $\V'_n(U')\subset\V_n''(V)$ if $U'\in\U''_n(U)$. Such a choice is possible since
$$
\begin{aligned}
\sum_{U'\in\U''_n(U)}(\mu(U')+\e_U)&=
\mu(\textstyle{\bigcup}\U''_n(U))+|\U''(n)|\cdot\e_U\le\mu(U\setminus A_{k_{n-1}})+|\U'_n(U)|\cdot\e_U<\\
&<\lambda\big(V\setminus O_{\e_U}(B_{m_{n-1}})\big)\le\lambda\big(\textstyle{\bigcup}\V_n''(V)\big)
\end{aligned}
$$and
$$\sum_{U'\in\U'_n(U)}(\mu(U')+\e_U)=
\mu(U)+|\U'(n)|\cdot\e_U<\lambda(V).$$

For every set $U'\in\U'_n(U)$ with $\mu(U')>0$ we get
$$\mu(U')<\lambda\big(\textstyle{\bigcup}\V_n''(U')\big)<\mu(U')+\e_U\le\mu(U')+2^{-n}\mu(U').$$
Using the continuity of the measure $\mu$, we can choose a clopen partition $\{U_{V'}\}_{V'\in\V_n'(U')}$ of $U'$ such that
$\mu(U_{V'})<\lambda(V')<(1+2^{-n})\mu(U_{V'})$ for every set $V'\in\V_n(U')$.

If $\mu(U')=0$, then let $\{U_{V'}\}_{V'\in\V_n'(U')}$ be any partition of $U'$.

Consider the clopen partition
$$\U_n(U)=\bigcup_{U'\in\U'_n(U)}\{U_{V'}:V'\in\V_n'(U')\}$$of $U$, the family  $\V'_n(U)=\bigcup_{U'\in\U'_n(U)}\V_n'(U')$ and the bijective map $\xi'_{U,n}:\U_n(U)\to\V_n'(U)$ assigning to each set $U''\in\U_n(U)$ the unique set $V'\in\V_n'(V)$ such that $U''=U_{V'}$ where $V'\in \V'_n(U')$ and $U'\in\U'_n(U)$ is the set containing $U''$.

It is easy to see that the families $\U_n(U)$ and $\V_n'(V)$ satisfy the conditions $(1_{U,n})$, $(4_{U,n})$ and $(6_{U,n})$.

Let $\U_n=\bigcup_{U\in\U_{n-1}}\U_n(U)$, $\V_n'=\bigcup_{U\in\U_{n-1}}\V'_n(U)$ and $\xi'_n:\U_n\to\V_n'$ be the bijective map such that $\xi'_n|\U_n(U)=\xi'_{U,n}$ for every $U\in\U_{n-1}$.

Consider the subfamily $\U_n^0=\{U\in\U_n:\mu(U)=0\}\subset \U_n$ and for every set $U\in\U_n^0$ choose a clopen set $\xi_n(U)\subset \xi_n'(U)\setminus B_{m_{n-1}}$ of measure $\lambda(\xi_n(U))<1/(2^n\cdot|\U_n|)$.

Let $\xi_n(U)=\xi_n'(U)$ for every $U\in\U_n\setminus\U_n^0$ and put $\V_n=\{\xi_n(U):U\in\U_n\}$.

Let $p_n:\U_n\to\U_{n-1}$ be the function assigning to each set $U'\in\U_n$ the unique set $U\in\U_{n-1}$ such that $U'\subset \U_n(U)$ and hence $U'\subset U$. By analogy define the map $\pi_n:\V_n\to\V_{n-1}$ assigning to each set $V'\in\V_n$ the unique set $V\in\V_{n-1}$ such that $V'\subset V$. It is easy to check that the families $\U_n$, $\V_n$, and the functions $\xi_n$, $p_n$ and $\pi_n$ satisfy the conditions $(1_n)$--$(6_n)$.

Using the equality   $\lambda(\bigcup_{m\in\w}B_m)=\lambda(Y)$ and the conditions $(4_n)$ and $(5_n)$, we can find a number $m_n>m_{n-1}$ such that the condition $(7_n)$ holds.

Using the equality $\mu(\bigcup_{k\in\w}A_k)=\mu(X)$ and taking into account that $\lambda(V\setminus B_{m_n})>0$ for every clopen set $V\in\V_n$, we can find a number $k_n>k_{n-1}$ satisfying the condition $(8_n)$. This completes the inductive construction.
\smallskip

Now define the map $f:X\to Y$ assigning to each point $x\in X$ the unique point of the intersection $\bigcap\{\xi_n(U):x\in U\in\U_n,\;n\in\w\}$. The conditions $(1_n)$--$(3_n)$, $n\in\IN$, imply that the map $f$ is a well-defined topological embedding, the conditions $(4_n)$ and $(5_n)$ for $n\in\IN$ imply that $f$ is measure-preserving, and the conditions $(6_n)$, $n\in\IN$, imply that $f(G_X)\subset G_Y$.
\end{proof}

\begin{lemma}\label{hom1} Let $(X,\mu)$ be a good Cantor measure space, $A$ be a closed nowhere dense subset and $B\subset X$ be a Borel subset of measure $\mu(B)>\mu(A)$ in $X$. Then there is a measure-preserving homeomorphism $h:X\to X$ such that $h(A)\subset B$.
\end{lemma}

\begin{proof} Using the fact that each open subset of $X$ contains a Cantor set of zero measure, we can construct a closed subset $\tilde A\subset X$ without isolated points which contains the set $A$ and has measure $\mu(\tilde A)=\mu(A)$. Then $(\tilde A,\mu|\tilde A)$ is a Cantor measure space. Fix any countable dense subset $Q\subset X$. Since $\mu(B\setminus Q)=\mu(B)>\mu(A)$, by the regularity of the measure $\mu$ there is a compact subset $K\subset B\setminus Q$ such that $\mu(K)>\mu(A)$. Replacing $K$ by the support of the measure $\mu|K$ we can assume that the measure $\mu|K$ is strictly positive. In this case $K$ has no isolated points and hence $(K,\mu|K)$ is a Cantor measure space. Let $G_A=\emptyset$ and $G_K$ be a dense $G_\delta$-set of measure $\mu(G_K)=0$ in $K$. By Lemma~\ref{haplyk}, there is a measure-preserving embedding $f:\tilde A\to K$ such that $h(G_A)\subset G_K$. By Lemma~\ref{exmes}, the embedding $f$ extends to a measure-preserving homeomorphism $h:X\to X$. It follows that $h(A)\subset h(\tilde A)=f(\tilde A)\subset K\subset B$.
\end{proof}

Lemma~\ref{hom1} implies the following ergodicity property of good Cantor measure spaces.

\begin{lemma}\label{ergo} If $A\subset X$ is a closed subset of positive measure in a good Cantor measure space $(X,\mu)$, then for any $\e>0$ there are homeomorphisms $h_1,\dots,h_n\in\HH_\mu(X)$ such that the set $B=\bigcup_{i=1}^nh_i(A)$ has measure $\mu(B)>\mu(X)-\e$.
\end{lemma}

\begin{proof} Using the continuity of the measure $\mu$, choose a finite disjoint family $\F$ consisting of closed nowhere dense subsets of $X$ such that $\mu(\bigcup\F)>\mu(X)-\e$ and $\mu(F)<\mu(A)$ for every $F\in\F$. For every $F\in\F$ apply Lemma~\ref{hom1} to find a homeomorphism $h_F\in\HH_\mu(X)$ such that $h_F(F)\subset A$.
Then the set $B=\bigcup_{F\in\F}h_U^{-1}(A)$ contains the set $\bigcup\F$ and hence has measure $\mu(B)\ge\mu(\bigcup\F)>\mu(X)-\e$.
\end{proof}

It is well-known that for any compact metric space $(X,d)$ the homeomorphism group $\HH(X)$ endowed with the compact-open topology is a Polish topological group. Moreover, the compact-open topology on $\HH(X)$ is generated by the complete metric
$$d_\HH(f,g)=\max_{x\in X}d(f(x),g(x))+\max_{x\in X}d(f^{-1}(x),g^{-1}(x))\mbox{ \ where $f,g\in\HH(X)$}.$$
For any Borel $\sigma$-additive measure $\mu$ on $X$ the subgroup $\HH_\mu(X)$ of measure-preserving homeomorphisms is a closed subgroup in $\HH(X)$.

\begin{lemma}\label{emb} Let $(X,\mu)$ be a good Cantor measure space, and $d$ be a metric generating the topology of $X$. Let $B\subset X$ be a Borel subset of measure $\mu(B)=\mu(X)$. Let $A\subset C$ be two closed nowhere dense subsets in $X$ such that $A\subset B$. For any $\e>0$  there exists a measure-preserving homeomorphism $h:X\to X$ such that $h|A=\id|A$, $ h(C)\subset B$, and $d_\HH(h,\id)\le\e$.
\end{lemma}

\begin{proof} Using the zero-dimensionality of $X\setminus A$ choose a disjoint cover $\U$ of $X\setminus A$ by clopen subsets of $X$ such that each set $U\in \U$ has $\diam(U)\le\min\{\e,\frac13 \min_{x\in U}d(x,A)\}$. The strict positivity of the measure $\mu$  implies that for every $U\in\U$ the set $U\cap C$ has measure $\mu(U\cap C)<\mu(U)=\mu(U\cap B)$. By Lemma~\ref{hom1}, there is a measure-preserving homeomorphism $h_U:U\to U$ such that $h_U(U\cap C)\subset  U\cap B$. Then the map $h:X\to X$ defined by the formula
$$h(x)=\begin{cases}x&\mbox{if $x\in A$}\\
h_U(x)&\mbox{if $x\in U$ for some $U\in\U$}
\end{cases}
$$is a measure preserving homeomorphism of $X$ such that $h|A=\id|A$, $h(C)\subset B$ and $\sup_{x\in X}d(h(x),x))\le\e$.
\end{proof}

Lemma~\ref{emb} implies the following its self-generalization.

\begin{lemma}\label{emb2} Let $(X,\mu)$ be a good Cantor measure space and $d$ be a metric generating the topology of $X$. Let $B\subset X$ be a Borel subset of measure $\mu(B)=\mu(X)$. For any $\e>0$,  homeomorphism $f\in\HH_\mu(X)$ and closed nowhere dense subsets $A\subset C$ in $X$ with $f(A)\subset B$, there exists a homeomorphism $g\in\HH_\mu(X)$ such that $g|A=f|A$, $g(C)\subset B$ and $d_\HH(f,g)<\e$.
\end{lemma}

\begin{proof} Given an $\e>0$ and a homeomorphism $f\in\HH_\mu(X)$, find a positive real number $\delta<\e$ such that for any points $x,y\in X$ with $d(x,y)<\delta$ we get $d(f^{-1}(x),f^{-1}(y))<\e$. Such number $\delta$ exists by the uniform continuity of the homeomorphism $f^{-1}$. Given any closed nowhere dense subsets $A\subset C$ in $X$ with $f(A)\subset B$, apply Lemma~\ref{emb} to the sets $C'=f(C)$ and $A'=f(A)\subset C'\cap B$ and
 find a homeomorphism $h\in\HH_\mu(X)$ such that $h|A'=\id|A'$, $h(C')\subset B$ and $d_\HH(h,\id)<\delta$. Then the homeomorphism $g=h\circ f$ has the required properties: $g|A=f|A$, $g(C)\subset B$ and $d_\HH(f,g)<\e$.
 \end{proof}

The following lemma will be derived from Lemma~\ref{emb2} by the standard back-and-forth technique.

\begin{lemma}\label{equiv} For any meager $F_\sigma$-sets $A,B\subset X$ of
measure $\mu(A)=\mu(B)=\mu(X)$ in a good Cantor measure space $(X,\mu)$ there is a measure-preserving homeomorphism $h\in\HH_\mu(X)$ such that $h(A)=B$.
\end{lemma}

\begin{proof} Fix any metric $d$ generating the topology of the compact metrizable space $X$.

Write the meager $F_\sigma$-sets $A$ and $B$ as unions $A=\bigcup_{n\in\w}A_n$ and $B=\bigcup_{n\in\w}B_n$ of increasing sequences $(A_n)_{n\in\w}$ and $(B_n)_{n\in\w}$ of compact subsets of $X$ with $A_0=B_0=\emptyset$. Put $f_0:X\to X$ be the identity homeomorphism of $X$, and $\tilde A_0=\tilde B_0=\emptyset$. Fix any $\e>0$. By induction, for every $n\in\IN$ we shall construct two compact sets $\tilde A_n,\tilde B_n\subset X$ and measure-preserving homeomorphisms $f_n,g_n:X\to X$  satisfying the following conditions:
\begin{enumerate}
\item[$(1_n)$] $\tilde A_n=A_n\cup g_{n-1}(\tilde B_{n-1})\subset A$;
\item[$(2_n)$] $f_n|g_{n-1}(\tilde B_{n-1})=g_{n-1}^{-1}|g_{n-1}(\tilde B_{n-1})$ and $f_n(\tilde A_n)\subset B$;
\item[$(3_n)$] $\tilde d(f_n,g^{-1}_{n-1})<\e/2^{n+1}$;
\item[$(4_n)$] $\tilde B_n=B_n\cup f_n(\tilde A_n)\subset B$;
\item[$(5_n)$] $g_n|f_n(\tilde A_n)=f_n^{-1}|f_n(\tilde A_n)$ and $g_n(\tilde B_n)\subset A$;
\item[$(6_n)$] $\tilde d(g_{n},f_n^{-1})<\e/2^{n+2}$.
\end{enumerate}

Assume that for some $n\in\IN$ and all $k<n$ we have constructed sets $\tilde A_k$, $\tilde B_k$, and measure-preserving homeomorphisms $f_k,g_k:X\to X$ satisfying the conditions $(1_k)$--$(6_k)$. We shall construct sets $\tilde A_n,\tilde B_n$, and homeomorphisms $f_n,g_n:X\to X$. By the inductive assumption $(5_{n-1})$, $g_{n-1}(\tilde B_{n-1})\subset A$. Since $\mu(B)=\mu(X)$, we can apply Lemma~\ref{emb2} and find a measure-preserving homeomorphism $f_n:X\to X$ such that $f_n|g_{n-1}(\tilde B_{n-1})=g_{n-1}^{-1}|g_{n-1}(\tilde B_{n-1})$, $f_n(A_n)\subset B$ and $\tilde d(f_n,f_{n-1})<\e/2^{n+1}$.

Put $\tilde B_n=B_n\cup f_n(A_n\cup\tilde A_{n-1})$. Applying Lemma~\ref{emb2}, we can find a measure-preserving homeomorphism $g_n:X\to X$ such that $g_n|f_n(A_n\cup\tilde A_{n-1})=f_n^{-1}|f_n(A_n\cup\tilde A_{n-1})$, $g_n(\tilde B_n)\subset A$, and $\tilde d(g_n,f_n^{-1})<\e/2^{n+2}$. Finally put $\tilde A_n=A_n\cup\tilde A_{n-1}\cup g_n(\tilde B_n)$ and observe that the sets $\tilde A_n$, $\tilde B_n$, and the homeomorphisms $f_n$ and $g_n$ satisfy the conditions $(1_n)$--$(6_n)$.

After completing the inductive construction, we obtain sequences of sets $\tilde A_n$, $\tilde B_n$, and homeomorphisms $f_n,g_n\in\HH_\mu(X)$ satisfying the conditions $(1_n)$--$(6_n)$ for all $n\in\IN$.

The conditions $(1_n)$, $(2_n)$, $(4_n)$, and $(5_n)$ for $n\in\IN$ imply that
$$f_n|\tilde A_{n-1}=g_{n-1}^{-1}|\tilde A_{n-1}=f_{n-1}|\tilde A_{n-1}\mbox{ \ and \ }
g_n|\tilde B_{n-1}=f_n^{-1}|\tilde B_{n-1}=g_{n-1}|\tilde B_{n-1}.$$

The conditions $(3_n)$, $(6_n)$, $n\in\IN$, imply that
$$d_\HH(f_n,f_{n-1})\le d_\HH(f_n,g_{n-1}^{-1})+d_\HH(g_{n-1}^{-1},f_{n-1})=
d_\HH(f_n,g_{n-1}^{-1})+d_\HH(g_{n-1},f^{-1}_{n-1})< \frac\e{2^{n+1}}+\frac\e{2^{n+1}}=\frac\e{2^n}.
$$and
$$d_\HH(g_n,g_{n-1})\le d_\HH(g_n,f_{n}^{-1})+d_\HH(f_{n}^{-1},g_{n-1})=d_\HH(g_n,f_{n}^{-1})+
d_\HH(f_{n},g^{-1}_{n-1})< \frac\e{2^{n+2}}+\frac\e{2^{n+1}}<\frac\e{2^n}.
$$
It follows that the sequence of measure-preserving homeomorphisms $(f_n)_{n\in\w}$ in the complete metric group $(\HH(X),d_\HH)$ is Cauchy and hence tends to a measure-preserving homeomorphism $f=\lim_{n\to\infty}f_n$ such that $d_\HH(f,\id)\le\sum_{n=1}d_\HH(f_n,f_{n-1})<\e$. By the same reason, the sequence $(g_n)_{n\in\w}$ tends to a measure-preserving homeomorphism $g=\lim_{n\to\infty}g_n$ such that $d_\HH(g,\id)<\e$. It follows that
$f\circ g|\tilde B_n=f\circ g_n|\tilde B_n=f_{n+1}\circ g_n|\tilde B_n=\id|\tilde B_n$ for every $n\in\IN$. The density of the union $B=\bigcup_{n\in\IN}\tilde B_n$ in $X$ implies that $f\circ g$ is the identity homeomorphism of $X$ and hence $g=f^{-1}$. It follows that
$f(\tilde A_n)=f_n(\tilde A_n)\subset B$ and $g(\tilde B_n)=g_n(\tilde B_n)\subset A$ for every $n\in\w$, which implies that $f(A)\subset B$ and $f^{-1}(B)=g(B)\subset A$ and hence $f(A)=B$.
\end{proof}

In the proof of the following lemma we shall use induction by the tree $\w^{<\w}=\bigcup_{n\in\w}\w^n$ consisting of finite sequences of non-negative integer numbers. So, we recall some standard facts on the structure of the tree $\w^{<\w}$.  For a finite sequence $s=(s_0,\dots,s_{n-1})\in\w^n\subset\w^{<\w}$ and a number $i\in\w$ by $s\hat{\,}i=(s_0,\dots,s_{n-1},i)$ we denote the {\em concatenation} of $s$ and $i$.
For a finite sequence $s=(n_0,\dots,n_{k-1})\in \w^k\subset \w^{<\w}$ by $|s|$ we shall denote its length $k$. For an infinite sequence $s\in\w^\w$ we put $|s|=\w$. For any (finite or infinite) sequence $s=(s_i)_{i<|s|}\in\w^{<\w}\cup\w^\w$ and a number $n\le |s|$ by $s|n=(s_0,\dots,s_{n-1})$ we denote the initial segment of $s$. The set $\w^{<\w}$ is endowed with the partial order $\le$ in which $s\le t$ iff $s$ coincides with the initial segment $t\big|n$ of $t$ for $n=|s|$. The partially ordered set $(\w^{<\w},\le)$ is a tree is the sense that it has the smallest element $\emptyset$ (the empty sequence) and for any $s\in\w^{<\w}$ the lower set ${\downarrow}s=\{t\in\w^{<\w}:t\le s\}=\{s|n:n\le|s|\}$ is well-ordered.

We recall that for a measure space $(X,\mu)$ by $\mathcal E$ we denote the $\sigma$-ideal generated by closed null subsets of $X$.

\begin{lemma}\label{analytic} If an analytic subset $A\subset X$ of a Cantor measure space $(X,\mu)$ is not contained in the $\sigma$-ideal $\mathcal E$, then $A$ contains a $G_\delta$-subset $G$ of $X$ such that $\mu(G)=0$ and the measure $\mu|\bar G$ is strictly positive.
\end{lemma}

\begin{proof} If $\mu(A)>0$, then by the regularity of the measure $\mu$, there is a compact subset $K\subset A$ of positive measure $\mu(K)$. Replacing $K$ by the support of the measure $\mu|K$, we can assume that the measure $\mu|K$ is strictly positive. Choose any dense $G_\delta$-set $G\subset K$ of measure $\mu(G)=0$ and observe that it has the required property: $\mu|\bar G$ is strictly positive.

Now assume that $\mu(A)=0$. The space $A$, being analytic, is the image of a Polish space $P$ under a surjective continuous map $f:P\to A$. Let $E\subset P$ be the set of points $x\in P$ having a neighborhood $U_x\subset P$ such that $f(U_x)\in\mathcal E$. It follows that $E$ is an open subset of $P$ such that $f(E)\in\mathcal E$ and $P\setminus E$ is a closed subset of $P$ such that for each non-empty open set  $U\subset P\setminus E$ the set $f(U)$ does not belong to the $\sigma$-ideal $\mathcal E$. Replacing $P$ by $P\setminus E$ and $A$ by $f(P\setminus E)$, we can assume that $E=\emptyset$, which means that for any non-empty open set $U\subset P$ the set $f(U)$ does not belong to the $\sigma$-ideal $\mathcal E$. Fix any complete metric $d_P$ generating the topology of the Polish space $P$ and any metric $d_X$ generating the topology of the Cantor space $X$.

By induction we shall construct a family $(U_s)_{s\in\w^{<\w}}$ of open subsets of $P$ satisfying the following conditions for every $s\in\w^{<\w}$:
\begin{enumerate}
\item[$(1_s)$] $\cl_P(U_{s\hat{\,}i}) \subset U_s$ for every $i\in\w$;
\item[$(2_s)$] $\cl_X(f(U_{s\hat{\,}i}))\cap\cl_X(f(U_{s\hat{\,}j}))=\emptyset$ for any distinct numbers $i,j\in\w$;
\item[$(3_s)$] $\max\{\diam(U_{s\hat{\,}i}),\diam(f(U_{s\hat{\,}i}))\}<1/2^{|s|+i}$ for every $i\in\w$;
\item[$(4_s)$] for the set $W_s=\bigcup_{i\in\w}f(U_{s\hat{\,}i})$ its closure $\overline{W}_s$ in $X$ has measure $\mu(\overline{W}_s)>0$.
\end{enumerate}

To start the inductive construction, put $U_\emptyset=P$. Assume that for some $s\in\w^{<\w}$ the open set $U_s\subset P$ has been constructed. By our assumption, the set $\overline{f(U_s)}$ has positive measure $\mu(\overline{f(U_s)})$. Since $\mu(A)=0$, there is a compact subset $K_s\subset \overline{f(U_s)}\setminus A$ of positive measure. Choose a sequence of pairwise distinct points $(y_{s,i})_{i\in\w}$ in $f(U_s)$ such that the space $\{y_{s,i}\}_{n\in\w}$ is discrete and contains the compact set $K_s$ in its closure. For every point $y_{s,i}$ fix a neighborhood $O(y_{s,i})\subset X$ such that $\diam(O(y_{s,i}))<1/2^{|s|+i}$ and the family $\{\cl_X(O(y_{s,i}))\}_{i\in\w}$ is disjoint. For every $i\in\w$ choose a point $x_{s,i}\in U_s$ with $f(x_{s,i})=y_{s,i}$ and choose a neighborhood $U_{s\hat{\,}i}\subset P$ of $x_{s,i}$ such that $f(U_{s\hat{\,}i})\subset O(y_{s,i})$ and $\diam(U_{s\hat{\,}i})<1/2^{|s|+i}$.
It is clear that the family $(U_{s\hat{\,}i})_{i\in\w}$ satisfies the inductive conditions $(1_s)$--$(4_s)$.

After completing the inductive construction, consider the map $g:\w^\w\to P$ assigning to each sequence $s\in\w^\w$ the unique point
of the intersection $\bigcap_{n\in\w}U_{s|n}$. The conditions $(3_s)$, $s\in\w^{<\w}$, imply that the map $g$ is continuous and the conditions $(2_s)$, $s\in\w^{<\w}$, imply that $f\circ g:\w^\w\to A$ is a topological embedding. Then $G=f\circ g(\w^\w)\subset f(P)\subset A$ is a $G_\delta$-subset of $X$, which is homeomorphic to the Baire space $\w^\w$. It remains to prove that any non-empty open set $U\subset \bar G$ has positive measure $\mu(U)$. By the regularity of the compact space $\bar G$, there is a non-empty open set $W\subset G$ such that $\cl_X(W)\subset U$. Since $f\circ g:\w^\w\to G$ is a homeomorphism, we can find a sequence $s\in\w^{<\w}$ such that for the open set $W_s=\{t\in\w^\w:s\le t\}$ we get $f\circ g(W_s)\subset W$. For every $i\in\w$ fix a point $z_i\in W_{s\hat{\,}i}\subset W_s$ and observe that $f\circ g(z_i)\subset f(U_{s\hat{\,}i})\subset O(y_{s,i})$. Since $\diam O(y_{s,i})<1/2^{|s|+i}$, the sets $\{f\circ g(z_i):i\in\w\}$ and $\{y_{s,i}:i\in\w\}$ has the same set of accumulation points in $X$. Consequently, $K_s\subset \cl_X(\{f\circ g(z_i):i\in\w\})\subset \cl_X(f(W_s))\subset \cl_X(W)\subset U$ and the set $U$ has measure $\mu(U)\ge\mu(K_s)>0$. This means that the measure $\mu|\bar G$ is strictly positive.
\end{proof}

\section{Some properties of invariant $\sigma$-ideals on good Cantor measure spaces}\label{s3}

In this section we shall prove some extremal properties of the standard invariant $\sigma$-ideals $\M$, $\mathcal N$, $\mathcal M\cap\mathcal N$, $\mathcal E$ on good Cantor measure spaces.

\begin{lemma}\label{l:C} Let $\I$ be an invariant ideal with analytic base on a good Cantor measure space $(X,\mu)$. If $\I\not\subset[X]^{\le\w}$, then $\I$ contains any closed null subset $E$ of $X$.
\end{lemma}

\begin{proof} It follows from $[X]^{\le\w}\not\subset\I$ that the ideal $\I$  contains an uncountable set $A\subset X$. Since $\I$ has an analytic base, we can assume that the set $A$ is analytic. By \cite[29.1]{Ke}, the uncountable analytic space $A$ contains a Cantor set $C\subset A\in\I$. Replacing $C$ by a smaller Cantor set, we can assume that $\mu(C)=0$.

Since $C$ contains a topological copy of each zero-dimensional compact metrizable space \cite[7.8]{Ke}, for any closed null set $E\subset X$, there is a topological embedding $f:E\to C$. Since $E$ and $C$ have measure zero, the embedding $f$ is measure-preserving. By Lemma~\ref{exmes}, the map $f$ extends to a measure-preserving homeomorphism $\bar f:X\to X$. The inclusion $f(E)\subset C\subset A\in\I$ and the invariance of the ideal $\I$ imply that $E\in\I$. \end{proof}

\begin{lemma}\label{l:E} Let $\I$ be an invariant ideal with analytic base on a good Cantor measure space $(X,\mu)$. If $\I\not\subset\E$, then $\I$ contains any nowhere dense null subset $N$ of $X$.
\end{lemma}

\begin{proof} Given any nowhere dense null subset $N\subset X$, consider the closure $\bar N$ of $N$ in $X$ and observe that $\bar N$ is nowhere dense in $X$ and hence has measure $\mu(\bar N)<\mu(X)$ by the strict positivity of the measure $\mu$. Replacing $N$ by a larger nowhere dense null set, we can assume that $\bar N$ has no isolated points and hence is a Cantor set in $X$. The set $N$ has measure $\mu(N)=0$ and hence is contained in a $G_\delta$-set $\tilde N\subset \bar N$ of measure $\mu(\tilde N)=0$.

Since the ideal $\I\not\subset\E$ has analytic base, there exists an analytic set $A\in\I\setminus\mathcal E$. By Lemma~\ref{analytic}, the set $A$ contains a $G_\delta$-subset $G'\subset X$ such that the measure $\mu|\bar G'$ is strictly positive. By Lemma~\ref{ergo}, there are homeomorphisms $h_1,\dots,h_n\in\HH_\mu(X)$ such that the set $\bar G=\bigcup_{i=1}^nh_i(\bar G')$ has measure $\mu(\bar G)>\mu(\bar N)$.
It follows that the set $\bar G$ coincides with the closure of the $G_\delta$-set $G=\bigcup_{i=1}^nh_i(G')$, which belongs to the ideal $\I$ by the invariance of $\I$. The strict positivity of the measure $\mu|\bar G'$ implies the strict positivity of the measure $\mu|\bar G$. By Lemmas~\ref{haplyk} and \ref{exmes}, there is a homeomorphism $h\in\HH_\mu(X)$ such that $h(\bar N)\subset\bar G$ and $h(\tilde N)\subset G$. Then $N\subset\tilde N\subset h^{-1}(G)\in\I$ by the invariance of the ideal $\I$.
\end{proof}

We shall say that an ideal $\I$ on a measurable space $(X,\mu)$ has {\em measurable base} if $\I$ has a base consisting of $\mu$-measurable subsets. It follows from the measurability of analytic sets \cite[21.10]{Ke} that each ideal with analytic base on a metrizable measurable space $(X,\mu)$ has measurable base.

\begin{lemma}\label{l:N} Let $\I$ be an invariant ideal with analytic (more generally, measurable) base on a good Cantor measure space $(X,\mu)$. If $\I\not\subset\mathcal N$, then $\I$ contains each nowhere dense subset $N$ of $X$.
\end{lemma}

\begin{proof} By our assumption, the ideal $\I\notin\mathcal N$ contains a measurable subset $A\subset X$ of positive measure $\mu(A)$. By the regularity of the measure $\mu$, the set $A$ contains a compact subset $K\subset A$ of positive measure $\mu(K)$.

Fix any nowhere dense subset $N\subset X$ and observe that its closure $\bar N$ has measure $\mu(\bar N)<\mu(X)$. By Lemma~\ref{ergo}, there are homeomorphisms $h_1,\dots,h_n\in\HH_\mu(X)$ such that the compact set $B=\bigcup_{i=1}^nh_i(K)$ has measure $\mu(B)>\mu(\bar N)$. The invariance of the ideal $\I\ni K$ guarantees that $B\in\I$. By Lemma~\ref{hom1}, there is a homomorphism $h\in\HH_\mu(X)$ such that $h(\bar N)\subset B$. Then the nowhere dense subset $N\subset \bar N\subset h^{-1}(B)\in\I$ belongs to the ideal $\I$ by the invariance of $\I$.
\end{proof}

We shall say that an ideal $\I$ on a topological space $X$ has {\em BP-base} if $\I$ has a base consisting of sets with Baire Property in $X$. We recall that a subset $A$ of a topological space $X$ has the {\em Baire property} if for some open set $U\subset X$ the symmetric difference $A\triangle U=(A\setminus U)\cup(U\setminus A)$ is meager in $X$. By Ulam-Sierpi\'nski Theorem \cite[21.6]{Ke}, any analytic subset of a metrizable space $X$ has the Baire property in $X$. Consequently, any ideal with analytic base on a metrizable space has BP-base.

\begin{lemma}\label{l:M} Let $\I$ be an invariant ideal with BP-base on a good Cantor measure space $(X,\mu)$. If $\I\not\subset\mathcal M$, then $\mathcal N\subset \I$.
\end{lemma}

\begin{proof} Fix any set $A\in\I\setminus\M$. Since $\I$ has BP-base, we can replace $A$ by a larger set with Baire Property in $\I$ and assume that $A$ has the Baire property. Then for some non-empty open set $U\subset X$ the symmetric difference $A\triangle U$ is meager in $X$ and hence is contained in a meager $F_\sigma$-subset $F\subset X$. Then $U\setminus F\subset A$ is a $G_\delta$-subset of $X$, dense in the open set $U$. The transitivity of the action of the group $\HH_\mu(X)$ on $X$ implies that $\{h(U):h\in\HH_\mu(X)\}$ is an open cover of $X$. By the compactness of $X$, this cover has a finite subcover $\{h_1(U),\dots,h_n(U)\}$. Then $G=\bigcup_{i=1}^nh_i(U\setminus F)$ is a dense $G_\delta$-set in $X$, which belongs to the ideal $\I$ by the invariance of $\I$. Using the continuity and regularity of the measure $\mu$, we can choose a dense $G_\delta$-set $\tilde G\subset G$ of measure $\mu(\tilde G)=0$.

Now fix any null subset $N\in\mathcal N$ in $X$. By the continuity and regularity of the measure $\mu$ the null set $N$ can be enlarged to a null dense $G_\delta$-set $\tilde N\subset X$. It follows that $X\setminus\tilde G$ and $X\setminus\tilde N$ are two meager $F_\sigma$-sets in $X$ of measure $\mu(X\setminus\tilde G)=\mu(X)=\mu(X\setminus\tilde N)$. By Lemma~\ref{equiv}, there is a homeomorphism $h\in\HH_\mu(X)$ such that $h(X\setminus \tilde G)=X\setminus\tilde N$. Then $N\subset \tilde N=h(\tilde G)\subset h(G)\in\I$ by the invariance of the ideal $\I$.
\end{proof}

\begin{lemma}\label{l:MN} Let $\I$ be an invariant $\sigma$-ideal with measurable BP-base on a good Cantor measure space $(X,\mu)$. If $\I\not\subset \M$ and $\I\not\subset\mathcal N$, then $\I=[X]^{\le \mathfrak c}$.
\end{lemma}

\begin{proof} By Lemmas~\ref{l:M} and \ref{l:N}, $\I\not\subset\M$ and $\I\not\subset\mathcal N$ imply $\M\cup\mathcal N\subset\I$. The continuity and regularity of the measure $\mu$ implies that the space $X$ can be written as the union $X=M\cup N$ of a meager set $M$ and a null set $N$. Then $X\in \I$ being the union of two sets from the ideal $\I$.
\end{proof}

\section{Proof of Theorem~\ref{main}}\label{s4}

Let $\I$ be a $\sigma$-ideal with analytic base of a good Cantor measure space $(X,\mu)$.

If $\I$ contains only the empty set, then $\I=[X]^{\le0}$.

So, we assume that $\I$ contains some non-empty subset. In this case $\I$ contains some singleton $\{x\}\subset X$ and by the invariance of $\I$ and transitivity of the action of the group $\HH_\mu(X)$, the ideal $\I$ contains any singleton of $X$. Now the $\sigma$-additivity of $\I$ implies $[X]^{\le\w}\subset\I$. If $\I\subset[X]^{\le\w}$, then $\I=[X]^{\le\w}$ and we are done.

So, we assume that $\I\not\subset[X]^{\le\w}$. In this case $\mathcal E\subset\I$ by Lemma~\ref{l:C}. If $\I\subset\E$, then $\I=\E$.

So, we assume that $\I\not\subset\E$. In this case Lemma~\ref{l:E} implies that $\M\cap\mathcal N\subset\I$. If $\I\subset\M\cap\mathcal N$, then $\I=\M\cap\mathcal N$.

So, we assume that $\I\not\subset\mathcal M\cap\mathcal N$, which means that $\I\not\subset\M$ or $\I\not\subset\mathcal N$.

First we assume that $\I\not\subset\M$. In this case Lemma~\ref{l:M} implies that $\mathcal N\subset\I$. If $\I\subset\mathcal N$, then $\I=\mathcal N$ and we are done. If $\I\not\subset\mathcal N$, then $\I=[X]^{\le\mathfrak c}$ by Lemma~\ref{l:MN}.

Finally, assume that $\I\not\subset\mathcal N$. In this case Lemma~\ref{l:N} implies that $\M\subset \I$. If $\I\subset\M$, then $\I=\M$. If $\I\not\subset\M$, then $\I=[X]^{\le\mathfrak c}$ according to Lemma~\ref{l:MN}.

\end{document}